\documentclass[11pt]{article}
\usepackage{amssymb, amsthm ,amsmath, dsfont,graphics,graphicx,pspicture,graphpap}
\usepackage{hyperref}
\usepackage{color}
\usepackage{enumerate}

\def\lec{\lesssim}

\newcommand{\ps}[1]{\left( #1 \right)}

\newcommand{\av}[1]{\left| #1 \right|}
\def\ve{\epsilon}

\def\d{\partial}
\def\avint_#1{\mathchoice
{\mathop{\vrule width 6pt height 3 pt depth -2.5pt
\kern -8.8pt \intop}\nolimits_{#1}}%
{\mathop{\vrule width 5pt height 3 pt depth -2.6pt
\kern -6.5pt \intop}\nolimits_{#1}}%
{\mathop{\vrule width 5pt height 3 pt depth -2.6pt
\kern -6pt \intop}\nolimits_{#1}}%
{\mathop{\vrule width 5pt height 3 pt depth -2.6pt \kern -6pt
\intop}\nolimits_{#1}}}

\newcommand {\bt}{{\tilde{\beta}} }
\newcommand {\cM}{\mathcal {M}}

\newcommand{\Claim}{{\bf Claim}}

\newcommand {\side}{{\rm side}}
\newcommand{\diam}{{\rm diam}}
\newcommand{\dist}{{\rm dist}}

\newcommand{\vol}{{\rm {vol}}}

\newcommand {\bR}{\mathbb {R}}

\newcommand {\bN}{\mathbb{N}}

\newcommand{\eqn}[1]{\eqref{e:#1}}

\def\co{\overline{\mbox{co}}}

\def\md{{\rm md}}

\theoremstyle{plain}
\newtheorem{theorem}{Theorem}[section]
\newtheorem{lemma}[theorem]{Lemma}
\newtheorem{proposition}[theorem]{Proposition}

\newtheorem{remark}[theorem]{Remark}

\numberwithin{equation}{section}

\begin{document}

\title{A quantitative metric differentiation theorem}

\author{
Jonas Azzam\footnote{
	Dep. of Mathematics,
	Univ. of Washington,
	Box 354350,
	Seattle, WA 98195-4350,
	U.S.A.
} \ and Raanan Schul\footnote{
	Dep. of Mathematics,
	Stony Brook Univ.,
	Stony Brook, NY 11794-3651,
	U.S.A.}}
\maketitle

\begin{abstract}
The purpose of this note is to point out a simple consequence of some earlier work of the authors, ``Hard Sard: Quantitative implicit function and extension theorems for Lipschitz maps". 
For $f$, a Lipschitz function from a Euclidean space into a metric space, we give quantitative estimates for how often the pullback of the metric under $f$ is approximately a seminorm.
This is a quantitative version of Kirchheim's metric differentiation result from 1994. Our result  is in the form of a Carleson packing condition.\\
{\bf Mathematics Subject Classification (2000):} 26A16, 54E40, \\ 
{\bf Keywords:} Lipschitz map, metric differential, differentiability of Lipschitz maps, quantitative differentiation, Rademacher.
%
%
%
%
\end{abstract}

\section{Introduction}
Consider a function from Euclidean space into a metric space,  $f:\bR^{n}\rightarrow \cM$. Without loss of generality, we will assume $\cM=\ell^{\infty}$ which will ease some notation below.
Let $ \Delta(\bR^{n})$ be a collection of dyadic cubes  in $\bR^n$
and  $\Delta(R)=\{Q\in   \Delta(\bR^{n}),\ Q\subset R\}$.
For $Q\in \Delta(\bR^{n})$, 
$\side(Q)$ denotes its sidelength, $x_{Q}$ its center, and $3Q$ the cube with the same center but 3 times the sidelength. Define
\[\md(Q):=\frac1{\side(Q)}\inf_{||\cdot||} \sup_{x,y\in Q} \bigg| |f(x)-f(y)|-\|x-y\| \bigg|\]
where the infimum is taken over all seminorms $\|\cdot\|$ on $\bR^{n}$ {and $|\cdot|$ is the $\ell_{\infty}$ norm. }
A function $f$ as above is said to be $L$-Lipschitz if for all $x,y\in \bR^{n}$,
$$ \dist(f(x),f(y))\leq L|x-y|\,.$$
Our main theorem in this note is that the pull-back of of the distance function on $\cM$ under an $L$-Lipschitz function $f:\bR^{n}\rightarrow \cM$ is well approximated by a seminorm on most scales and in most locations. 
\begin{theorem}\label{t:main-theorem}
Let $f:\bR^{n}\rightarrow \cM$ be an $L$-Lipschitz function. Let $\delta>0$. Then
for each $R\in\Delta(\bR^{n})$
\[\sum\{\vol(Q):Q\in\Delta(R), \md(3Q)>\delta L\}\leq
 C_{\delta,n} \cdot \vol(R)\,.\]
The  constant $C_{\delta,n}$ does not  depend on the metric space $\cM$
or the function $f:\bR^n\to\cM$.
\end{theorem}

An estimate as above is typically called a Carleson packing condition in harmonic 
analysis or geometric measure theory (c.f. \cite{DS-analysis-of-and-on}). The name stems from the fact that  a Carleson packing on a collection of cubes usually 
arises from the stronger property of the existence of a Carleson measure on $[0,1]^{n}\times (0,\infty)$ (see \cite{BAF} or \cite{Big-Stein}). For example, Dorronsoro established in \cite{Dorronsoro-potential-spaces} that if $f$ is Lipschitz and
\[\Omega_{f}(x,t):=\inf \left\{ \avint_{B(x,t)} \frac{|f(y)-A(y)|}{t}dy: \;\; A:[0,1]^{n}\rightarrow \bR^{m}\mbox{ is affine}\right\}\]
then $\Omega_{f}(x,t)^{2}\frac{dxdt}{t}$ defines a Carleson measure. 
This implies Theorem \ref{t:main-theorem} in the case that $f$ maps into a Euclidean space and both its domain and range are equipped with the Euclidean norm. In fact, more is true: the theorem still holds if we replace $\md(Q)$ with $\alpha(Q)$ where $\alpha(Q)\side(Q):=\inf||f-A||_{L^{\infty}(Q)}$ and the infimum is taken over all affine mappings $A$ between the domain and range of $f$. In the argot of \cite{DS-analysis-of-and-on}, Theorem \ref{t:main-theorem} is the {\it weak approximation of Lipschitz functions by affine maps}, or the WALA property, which holds more generally for Lipschitz functions from uniformly rectifiable sets into a Euclidean space. For more information about Carleson packing conditions, the WALA property, and their applications to geometric measure theory, see \cite{Jones-TSP, DS,DS-analysis-of-and-on,DS00-regular-mappings,hardsard}.

Assaf Naor and Sean Li have recently developed estimates similar to that 
of Theorem 1.1 to estimate the size of the largest cube for which a Lipschitz 
map of a function into a super-reflexive space is approximately affine \cite{Li-Naor11}, strengthening an earlier infinitesimal  result of Bates, Johnson, Lindenstrauss, Preiss, and Schechtman \cite{BJLPS99}. Other authors who have been concerned with 
quantitative or coarse differentiation results include Jeff Cheeger, Alex Eskin, David Fisher, Irene Peng, and others (c.f. \cite{EFW07} and the references 
therein). 

There are many other references of this type which we 
omit here. A common thread to these is that one may view such statements as 
coarse or quantitative versions of Rademacher's theorem, which can be 
used to find a large scale where the function is approximately linear, or sufficiently regular in the 
appropriate sense. A Rademacher-type theorem for Lipschitz functions from Euclidean space into a general metric space was established by Kirchheim in terms of metric differentials.

\begin{theorem}\cite{Kirchheim-Rademacher}
If $f:[0,1]^{n}\rightarrow \cM$ is Lipschitz, where $\cM$ is a metric space, then for almost every $z\in [0,1]^{n}$, there is a seminorm $\|\cdot\|_{z}$ on $\bR^{n}$ such that 
\[ \bigg||f(x)-f(y)|-\|x-y\|_{z}\bigg| = o(|x-z|+|y-z|).\]
\end{theorem}
The seminorm $\|\cdot\|_{z}$ is called the {\it metric differential} of $f$ at $z$.
In other words, for almost every point in the domain, the pullback of the metric under $f$ is well approximated by  some seminorm on the domain. Theorem \ref{t:main-theorem} now gives a more quantitative version of this result in the spirit of \cite{Dorronsoro-potential-spaces}. The work below  is a simple consequence of some techniques and ideas that were originally developed in and for  \cite{hardsard}. 

In Section \ref{s:preliminaries}, we recall the definition of $\bt$-numbers from \cite{Schul-lip-bilip} and the relevant lemmas from \cite{hardsard}. In Section \ref{s:proof-of-main-theorem}, we prove the main theorem.\\

\noindent {\bf Acknowledgements.} We thank John Garnett for his helpful discussions, Sean Li for pointing out \cite{Li-Naor11}, and Matthew Badger for carefully reading a draft. Jonas Azzam was supported by RTG grant DMS-0838212. Part of this note was written while Jonas Azzam was visiting the Mathematical Sciences Research Institute. Raanan Schul was supported by a fellowship from the Alfred P. Sloan Foundation and by NSF grant DMS 1100008.

\section{Preliminaries}
\label{s:preliminaries}

For $Q\in \Delta(\bR^n)$, let
$Q^{N}$ denotes the unique dyadic cube containing $Q$ of sidelength $2^{N}\side(Q)$. For functions $A(t_{1},...,t_{k})$ and $B(t_{1},...,t_{k})$, we write $A\lec B$ if there is a constant $C$ (independent of $(t_{1},...t_{k})$) such that $A\leq CB$. We will also write $A\lec_{t_{i}}B$ if the implied constant $C$ depends on $t_{i}$. 

\subsection{$\bt$-numbers}

For a Lipschitz function $f:[0,1]^{n}\rightarrow \cM$, define
\begin{equation*}
 \d^{f}(x,y,z):=|f(x)-f(y)| + |f(y)-f(z)| -|f(x)-f(z)|.
\end{equation*}
For an interval $I=[a,b]\subset\bR$, let 
\begin{equation*}
\bt_{f}(I)^2\diam(I)=\diam(I)^{-3}\int_{x=a}^{x=b}\int_{y=x}^{y=b}\int_{z=y}^{z=b}  \d^{f}(x,y,z)dzdydx\,.
\end{equation*}
Identify $\bR$ with $\{\bR,0,...,0\}\subset \bR^n$,  and let $G_n$ be the group of all rotations of $\bR$ in $\bR^n$ equipped with its Haar measure $d\mu$. Let  $dx$ be the $n-1$ dimensional Lebesgue measure on $\bR^n \circleddash g\bR$, the orthogonal complement of $g\bR$ in $\bR^n$.
For a cube $Q\in \bR^n$, define
the quantity $\bt^{(n)}_{f}(Q)$ by
\begin{eqnarray*}
&&\bt^{(n)}_f(Q)^2\side(Q)^{n-1}=\\
&&\phantom{xxx}
\int_{g\in G_n}\int_{x\in \bR^n \circleddash g\bR} \chi_{\{|(x+g\bR)\cap 7Q|\geq \side(Q)\}}\bt((x+g\bR)\cap 7Q)^2dxd\mu(g)\,.
\end{eqnarray*}

\begin{theorem}[\cite{Schul-lip-bilip}]\label{t:TST}
For an $L$-Lipschitz function $f:[0,1]^{n}\rightarrow\cM$ and $N$ a fixed integer,
\begin{equation*}
\sum_{Q\in \Delta, Q\subseteq [0,1]^n} \bt_{f}(3Q^{N})^2\vol(Q)  \lec_{N, n} L\,.
\end{equation*}
\label{t:bt-sum}
\end{theorem}

\begin{remark}
If $\d^{f}(x,y,z)$ is small, this corresponds to $f(x),f(y)$ and $f(z)$ being close to lying on a geodesic. 
A crucial (and simple) observation in \cite{hardsard} was that if we define $\tilde{f}(x)=(x,f(x))$ as a map from $Q_{0}$ into $\bR^{n}\times\bR^{m}$, then $\d^{f}(x,y,z)$ being small corresponds to $f$ being approximately {\it affine} on $x,y,z$. 
Thus, for $\bt_{\tilde{f}}(Q)$ small enough, this gives that $f$ is approximately affine when restricted to $Q$. 
However, in the setting where $f$ maps into a metric space, if we interpret $\tilde{f}=(x,f(x))$ as a map from $\bR^{n}$ into $\bR^{n}\oplus \ell^{\infty}$, equipped with the metric $|(u,v)|=\sqrt{|u|_{2}^{2}+|v|_{\infty}^{2}}$, 
then $\bt_{\tilde{f}}(Q)$ being small corresponds to $f|_Q$ being approximately homogeneous on all lines. We will use this to prove Theorem \ref{t:main-theorem}. One should compare the proof with the usual proof of Rademacher's theorem, or the proof of Theorem 2 in \cite{Kirchheim-Rademacher}. 
\end{remark}

\subsection{Approximate Homogeneity}

Let $f:[0,1]^{n}\rightarrow\cM$ be $1$-Lipschitz. {Fix $\alpha\in (0,1)$ and $N\in \bN$ for now. The value for the constant $N$ is determined in Lemma \ref{l:shifting-lines}, with some dependencies, which are ultimately resolved in  Proposition \ref{p:beta-controls-md}, where $\alpha$ is also determined.}  
For $x,y\in 3Q$ with $|x-y|\geq \alpha\side(Q)$, define
\begin{equation}
\sigma(x,y)=
\inf_{x',y'\in L_{x,y}\cap 3Q^{N}\atop |x'-y'|\geq \alpha\side (Q)} \frac{|f(x')-f(y')|}{|x'-y'|}
\label{e:sigma-def}
\end{equation}
where $L_{x,y}$ is the line passing through $x$ and $y$. 

We recall some lemmas whose proofs can be found in \cite{hardsard}. Their proofs can be read independently of the rest of that paper.

\begin{lemma}\cite[Lemma 6.3 and following discussion]{hardsard}
Let $0<\alpha<1$ and $\ve'>0$ and $N\in\bN$. There is $\ve=\ve(N,\alpha,\ve')>0$ such that for any cube $Q$ with $\bt_{\tilde{f}}(3Q^N)<\ve$ 
we have
\begin{equation} \av{\frac{|f(x)-f(y)|}{|x-y|}-\sigma(x,y)}<\ve' 
\label{e:almost-isometric-on-lines}
\end{equation}
for all $x,y\in Q$ such that $|x-y|>\alpha\diam Q$.
\label{l:almost-isometric-on-lines}
\end{lemma}

In other words, if $\bt_{f}(3Q^N)$ is small, then $f$ is close to being homogeneous on lines. We remark that the constant $N$ here is only needed for consistency with the next lemma, Lemma \ref{l:shifting-lines}.

\begin{lemma}\cite[Corollary 6.5]{hardsard}
 For any $\rho>0$, there is $N=N(\rho,n)$ and $\ve=\ve(\alpha,\rho)$ such that if $\bt(3Q^{N})<\ve$, then 
if $x,y\in Q$ { are such that} $|x-y|\geq \alpha\side(Q)$, and { $z\in \bR^{n}$ is such that} $|z|\leq \side(Q)$,
\begin{equation}
\sigma(x+z,y+z)\leq \sigma(x,y)+\rho
\label{e:shift-sigma}
\end{equation}
and 
\begin{equation}
\bigg| |f(x+z)-f(y+z)| - |f(x)-f(y)|\bigg|<\rho\side(Q).
\label{e:shift-f}
\end{equation}
Moreover, if $\rho$ is small enough, depending on $\alpha$, and $x_{Q}=f(x_{Q})=0$, then for every $x,y\in Q$,
\begin{equation}
|f(x+y)|\leq |f(x)|+|f(y)|+\alpha\side(Q).
\label{e:t-ineq}
\end{equation}
\label{l:shifting-lines}
\end{lemma}

\begin{remark}
It is easy to show that the function $\sigma(x,y)$ is continuous on $\{(x,y)\in Q\times Q: |x-y|\geq \alpha \side Q\}$.
\label{r:continuous}
\end{remark}


\section{Proof of Theorem \ref{t:main-theorem}}
\label{s:proof-of-main-theorem}

We first recall an easy consequence of Carath\'eodory's convex hull theorem \cite[Theorem 1.2.3, p. 6]{Matousek}.
\begin{lemma}
Let $K$ be a set in $\bR^{n}$ and $\co K$ its closed convex hull. Then $\co K=\overline{\breve{K}}$, where
\[ \breve{K}:= \left\{\sum_{j=0}^{n}a_{j}x_{j}:x_{j}\in K,0\leq a_{j}\leq 1,\sum_{j=0}^{n} a_{j}=1\right\}.\]
That is, $\co K$ is the closure of all convex combinations of at most $n+1$ points from $K$.
\label{l:Carath}
\end{lemma}

%
%
%

We are now ready to state and prove the main tool.
\begin{proposition}\label{p:beta-controls-md}
Let $\delta>0$ and $f:[0,1]^{d}\rightarrow \cM$ be $1$-Lipschitz. Then there is $\ve>0$ and $N\in \bN$ such that if $\bt_{\tilde{f}}(3Q^{N})<\ve$ then $\md(Q)<\delta$. 
\end{proposition}
\begin{proof}
When rescaling the domain and range by the same factor, the Lipschitz constant of $f$, $\md$, and $\bt_{\tilde{f}}$ remain unchanged, thus we may assume that $Q$ centered at zero with $\side(Q)=1$ and $f(0)=0$. Fix $\alpha>0$, which will be specified later, and pick $\ve>0$ and $N$ so that the conclusions of Lemmas \ref{l:almost-isometric-on-lines} and \ref{l:shifting-lines} hold for $\ve'=\rho=\alpha$. 

Let $\sigma(x)=\sigma(0,x)$.  Set
$$C=\{x\in Q:\sigma(x)|x|\leq 1 \}$$
and
$$D=\{x\in Q:\sigma(x)|x|=1\mbox{ or } \sigma(x)=0\}\,.$$

For $t>0$ and $A\subset \bR^n$, define $tA=\{tx:x\in A\}$.

\noindent\Claim: For $\alpha'=2(n+1)\alpha$, 
\begin{equation} 
(1-\alpha')\co D\subseteq C\subseteq \co D.
\label{e:1-alpha'}
\end{equation}
The right containment is clear by the definition of $D$ and $\sigma(x)$ being constant on any given $1$-dimensional vector space (in particular, $C$ is a star-shaped subset of $\co D$).  To see the left containment, we proceed as follows.

{  Note that $C$ is closed because $\sigma$ is continuous (see Remark \ref{r:continuous}), thus by Lemma \ref{l:Carath}, it suffices to show that $(1-\alpha')\breve{D}\subseteq C$. In this vein, let $x\in (1-\alpha')D$, so that 
\[x=\sum_{j=0}^{n}a_{j}x_{j}\in \breve{D}, \mbox{ where }x_{j}\in (1-\alpha')D, \;\; a_{j}\geq0, \mbox{ and }\sum_{j=0}^{n}a_{j}=1.\] }


 Let $0\leq k\leq n$ be { the maximal integer} so that (after reordering the $x_j$'s) 
\begin{equation}
j\leq k\mbox{ implies }|x_{j}|\leq \alpha.
\label{e:j<k}
\end{equation}
Then by \eqn{t-ineq} and the definition of $\sigma$, 
\begin{multline*}
|x|\sigma(x) 
 \leq |f(x)-f(0)|
 =\av{f\ps{\sum_{j=0}^{n}a_{j}x_{j}}} 
\stackrel{\eqref{e:t-ineq}}{\leq} (n+1)\alpha+ \sum_{j=0}^{n}|f(a_{j}x_{j})| \\
 \stackrel{\eqref{e:j<k}}{\leq} (n+1)\alpha+(k+1)\alpha+\sum_{j=k+1}^{n}|f(a_{j} x_{j})|
  \stackrel{\eqref{e:almost-isometric-on-lines}}{\leq} 2(n+1)\alpha +\sum_{j=k+1}^{n}\sigma(a_{j} x_{j})|a_{j}x_{j}|\\
 \leq 2(n+1)\alpha+\sum_{j=k+1}^{n}a_{j}\sigma(x_{j})| x_{j}|
 \leq 2(n+1)\alpha+\sum_{j=0}^{n}a_{j}(1-\alpha')\leq 1
\end{multline*}
where in the last two inequalities, we used the facts that $\sigma\leq 1$ since $f$ is $1$-Lipschitz, $\sum_{j=0}^{n}a_{j}=1$, and $\alpha'=2(n+1)\alpha$. This proves the claim.\\

Let
\[\|x\|=\inf\{ t\geq 0: x\in t\co D\}.\]
By Kolmogorov's theorem \cite{kolmogoroff-seminorms} (see also \cite[Section II, p.20, Prop 23]{Bourbaki-TVS}),  since $\co D$ is convex and $-\co D=\co D$ (since $-D=D$), we have that $\|\cdot\|$ is a seminorm. By \eqn{1-alpha'},
\begin{equation}
\|x\|\leq \sigma(x)|x|\leq \frac{1}{1-\alpha'}\|x\|.
\label{e:sigma-and-norm}
\end{equation}
Let $x,y\in Q$. Since $f$ is $1$-Lipschitz, we have  $\sigma\leq 1$. By \eqn{sigma-and-norm}

\begin{equation} \|x-y\|\leq \sigma(x-y)|x-y|\leq |x-y|.
\label{e:|and||}
\end{equation}
Suppose $|x-y|>\alpha$, then

\begin{multline*}
|f(x)-f(y)|
\stackrel{\eqn{shift-f}}{\leq} |f(x-y)|+\alpha
\stackrel{\eqn{almost-isometric-on-lines}}{\leq} \sigma(x-y)|x-y|+2\alpha \\
\stackrel{\eqn{sigma-and-norm}}{\leq} \frac{1}{1-\alpha'}\|x-y\|+2\alpha\,,
\end{multline*}
and by 
Lemma \ref{l:shifting-lines}
\begin{equation}
\|x-y\| \stackrel{\eqn{sigma-and-norm}}{\leq} \sigma(x-y)|x-y| { \stackrel{\eqn{shift-sigma}}{\leq} (\sigma(x,y)+\rho)|x-y|\stackrel{\eqn{sigma-def}}{\leq}  |f(x)-f(y)| + \alpha\sqrt{n}\,. 		}
\end{equation}
{ (recall that $\side Q=1$ so that $\diam Q=\sqrt{n}$ and earlier we picked $\rho=\alpha$).}

If $|x-y|\leq \alpha$, then,  by $\eqn{|and||}$, $||x-y||\leq \alpha$. Furthermore, since $f$ is $1$-Lipschitz, $|f(x)-f(y)|\leq |x-y|\leq \alpha$. Thus, the difference between $||x-y||$ and $|f(x)-f(y)|$ is at most $\alpha$.

Combining the above estimates, 
and using the fact that $\|x-y\|\leq |x-y|\leq\diam Q =\sqrt{n}$, 
we have that for all $x,y\in Q$, 
\[\|x-y\|-\alpha\leq |f(x)-f(y)| \leq \frac{\|x-y\|}{1-\alpha'}+(2+\sqrt{n})\alpha  \leq \|x-y\|+\frac{\sqrt{n}\alpha'}{1-\alpha'}+(2+\sqrt{n})\alpha .\]
By choosing  $\alpha$ so that $\frac{\sqrt{n}\alpha'}{1-\alpha'}+2\alpha  { + \sqrt{n}\alpha} <\delta$ (recall that $\alpha'=2(n+1)\alpha$), we conclude the proof.

\end{proof}

\begin{proof}[Proof of Theorem \ref{t:main-theorem}]
Let $\delta>0$ be given and $R\in\Delta$. By rescaling and translating the domain, we may assume $R=[0,1]^{n}$. Note that $\frac{\md }{L}$ is invariant under such a transformation. By rescaling the metric, we may also assume that the Lipschitz constant of $f$ is $1$ and, again, this will not affect 
$\frac{\md }{L}$.
Then if $\ve=\ve(\delta)$ is as in the previous proposition, we have by Theorem \ref{t:TST}
\begin{multline}
\sum\{\vol(Q):Q\in \Delta([0,1]^n), \md(Q)>\delta\}\\
\leq \sum\{\vol(Q):Q\in \Delta([0,1]^n), \bt_{\tilde{f}}(3Q^{N})>\ve\} \\
\leq \frac{1}{\ve^{2}}\sum_{Q\in \Delta([0,1]^n)}\bt_{\tilde{f}}(3Q^{N})^{2}\vol(Q)\lec_{N,n}\frac{1}{\ve^{2}}\lec_{\delta}1.
\label{e:almost}
\end{multline}
By the standard $\frac{1}{3}$-trick (c.f. \cite[p. 339-340]{Okikiolu-TST}), {any cube $3Q$, 
of $\side(Q)\leq \frac16$
is contained in a cube of the form $R+v$ where $v\in \{0,\pm\frac{1}{3}\}^{n}$, $R\in \Delta([0,1]^n)$, and $\side(R)\lec_{n} \side(Q)$. Note that $\md(3Q)\lec_{n} \md(R+v)$ and $\vol (Q)\sim_{n} \vol (R)$.  Hence, we can apply \eqn{almost} to $f$ with respect to each of these translated grids to obtain
\begin{multline*}
\sum\{\vol(3Q):Q\in \Delta([0,1]^n), \md(Q)>\delta\}\\
\lec_{n} 
1+ \sum_{v\in \{0,\pm \frac{1}{3}\}^{n}}\sum\{\vol(R):R\in \Delta([0,1]^n), \md(R+v)>\delta\}\\
\lec_{n,\delta}  
1+ \sum_{v\in \{0,\pm \frac{1}{3}\}^{n}}1 \lec_{n} 1
\end{multline*}
}

\end{proof}

\bibliographystyle{amsalpha}
\bibliography{quant-kirch-ref}

\newcommand{\etalchar}[1]{$^{#1}$}
\providecommand{\bysame}{\leavevmode\hbox to3em{\hrulefill}\thinspace}
\providecommand{\MR}{\relax\ifhmode\unskip\space\fi MR }
\providecommand{\MRhref}[2]{%
  \href{http://www.ams.org/mathscinet-getitem?mr=#1}{#2}
}
\providecommand{\href}[2]{#2}
\begin{thebibliography}{BJL{\etalchar{+}}99}

\bibitem[AS12]{hardsard}
J.~Azzam and R.~Schul, \emph{Hard {S}ard: Quantitative implicit function and
  extension theorems for {L}ipschitz maps}, Geometric and Functional Analysis
  (2012).

\bibitem[BJL{\etalchar{+}}99]{BJLPS99}
S.~Bates, W.~B. Johnson, J.~Lindenstrauss, D.~Preiss, and G.~Schechtman,
  \emph{Affine approximation of {L}ipschitz functions and nonlinear quotients},
  Geom. Funct. Anal. \textbf{9} (1999), no.~6, 1092--1127. \MR{1736929
  (2000m:46021)}

\bibitem[Bou87]{Bourbaki-TVS}
N.~Bourbaki, \emph{Topological vector spaces. {C}hapters 1--5}, Elements of
  Mathematics (Berlin), Springer-Verlag, Berlin, 1987, Translated from the
  French by H. G. Eggleston and S. Madan. \MR{910295 (88g:46002)}

\bibitem[Dor85]{Dorronsoro-potential-spaces}
J.~R. Dorronsoro, \emph{A characterization of potential spaces}, Proc. Amer.
  Math. Soc. \textbf{95} (1985), no.~1, 21--31. \MR{796440 (86k:46046)}

\bibitem[DS91]{DS}
G.~David and S.~Semmes, \emph{Singular integrals and rectifiable sets in {${\bf
  R}^n$}: {B}eyond {L}ipschitz graphs}, Ast\'erisque (1991), no.~193.
  \MR{1113517 (92j:42016)}

\bibitem[DS93]{DS-analysis-of-and-on}
\bysame, \emph{Analysis of and on uniformly rectifiable sets}, Mathematical
  Surveys and Monographs, vol.~38, American Mathematical Society, Providence,
  RI, 1993. \MR{1251061 (94i:28003)}

\bibitem[DS00]{DS00-regular-mappings}
\bysame, \emph{{Regular mappings between dimensions}}, Publ. Mat. \textbf{44}
  (2000), 369--417.

\bibitem[EFW07]{EFW07}
Alex Eskin, David Fisher, and Kevin Whyte, \emph{Quasi-isometries and rigidity
  of solvable groups}, Pure Appl. Math. Q. \textbf{3} (2007), no.~4, part 1,
  927--947. \MR{2402598 (2009b:20074)}

\bibitem[Gar07]{BAF}
J.~B. Garnett, \emph{Bounded analytic functions}, revised first ed., Graduate
  Texts in Mathematics, vol. 236, Springer, New York, 2007. \MR{2261424
  (2007e:30049)}

\bibitem[Jon90]{Jones-TSP}
P.~W. Jones, \emph{Rectifiable sets and the traveling salesman problem},
  Invent. Math. \textbf{102} (1990), no.~1, 1--15. \MR{1069238 (91i:26016)}

\bibitem[Kir94]{Kirchheim-Rademacher}
B.~Kirchheim, \emph{Rectifiable metric spaces: local structure and regularity
  of the {H}ausdorff measure}, Proc. Amer. Math. Soc. \textbf{121} (1994),
  no.~1, 113--123. \MR{1189747 (94g:28013)}

\bibitem[Kol34]{kolmogoroff-seminorms}
A.~Kolmogoroff, \emph{Zur normierbarkeit eines allgemeinen topologischen
  linearen {R}aumes}, Studia Math. \textbf{5} (1934), 29--33.

\bibitem[LN]{Li-Naor11}
S.~Li and A.~Naor, \emph{Discretization and affine approximation in high
  dimensions}, preprint.

\bibitem[Mat02]{Matousek}
J.~Matousek, \emph{Lectures on discrete geometry}, Springer-Verlag, New York,
  2002.

\bibitem[Oki92]{Okikiolu-TST}
K.~Okikiolu, \emph{Characterization of subsets of rectifiable curves in {${\bf
  R}^n$}}, J. London Math. Soc. (2) \textbf{46} (1992), no.~2, 336--348.
  \MR{1182488 (93m:28008)}

\bibitem[Sch09]{Schul-lip-bilip}
R.~Schul, \emph{Bi-{L}ipschitz decomposition of {L}ipschitz functions into a
  metric space}, Rev. Mat. Iberoam. \textbf{25} (2009), no.~2, 521--531.
  \MR{2554164}

\bibitem[Ste93]{Big-Stein}
Elias~M. Stein, \emph{Harmonic analysis: real-variable methods, orthogonality,
  and oscillatory integrals}, Princeton Mathematical Series, vol.~43, Princeton
  University Press, Princeton, NJ, 1993, With the assistance of Timothy S.
  Murphy, Monographs in Harmonic Analysis, III. \MR{1232192 (95c:42002)}

\end{thebibliography}

\end{document}